\documentclass[pdflatex,sn-mathphys-num]{sn-jnl}% Math and Physical Sciences Numbered Reference Style

\usepackage{graphicx}%
\usepackage{multirow}%
\usepackage{amsmath,amssymb,amsfonts}%
\usepackage{amsthm}%
\usepackage{mathrsfs}%
\usepackage[title]{appendix}%
\usepackage{xcolor}%
\usepackage{textcomp}%
\usepackage{manyfoot}%
\usepackage{booktabs}%
\usepackage{algorithm}%
\usepackage{algorithmicx}%
\usepackage{algpseudocode}%
\usepackage{listings}%
\usepackage{enumitem}
\newcommand{\D}{\mathbb D}

\newcommand{\C}{\mathbb C}
\newcommand{\Ocal}{\mathcal O}
\newcommand{\Mcal}{\mathcal M}
\newcommand{\dd}{\,\mathrm d}

\theoremstyle{thmstyleone}%
\newtheorem{theorem}{Theorem}
\newtheorem{proposition}[theorem]{Proposition}
\newtheorem{lemma}[theorem]{Lemma}

\theoremstyle{thmstyletwo}%

\newtheorem{remark}[theorem]{Remark}

\theoremstyle{thmstylethree}%
\newtheorem{definition}[theorem]{Definition}

\begin{document}

\title[Meromorphic Continuation]{Determinant Characteristics and Argument-Principle Certification for Visible Poles in Meromorphic Continuation}
\author[1]{\fnm{Xiaomei} \sur{Yang}}\email{yangxiaomath@swjtu.edu.cn}

\author*[2]{\fnm{Zhiliang} \sur{Deng}}\email{dengzhl@uestc.edu.cn}


\affil[1]{\orgdiv{School of Mathematics}, \orgname{Southwest Jiaotong University}, \orgaddress{\street{No. 999, Xi'an Road, Pidu District}, \city{Chengdu}, \postcode{611756}, \state{Sichuan}, \country{China}}}

\affil*[2]{\orgdiv{School of Mathematical Science}, \orgname{University of Electronic Science and Technology of China}, \orgaddress{\street{No.2006, Xiyuan Ave, West Hi-Tech Zone}, \city{Chengdu}, \postcode{611731}, \state{Sichuan}, \country{China}}}

\abstract{We study outward meromorphic continuation from circular boundary data in the
unit disk.  The unknown function is holomorphic in the unit disk and admits a
meromorphic continuation to a larger disk, where finitely many exterior simple
poles are superposed on an unknown holomorphic background.  The positive
Fourier coefficients of the boundary trace are Taylor coefficients at the
origin, and exterior poles generate a finite exponential-sum component in these
coefficients.  We introduce shifted determinant characteristics and prove that,
in the pure finite-pole model, the determinant for the correct order factors
exactly into a nonzero constant times the polynomial whose zeros are the
reciprocals of the exterior poles.  The same zero set is obtained for noiseless
equispaced discrete Fourier coefficients; sampling changes only the amplitudes
through an aliasing factor.  For data containing a holomorphic background,
discretization effects, and noise, roots of a single empirical determinant are
only candidate reciprocal poles.  We therefore propose a root-propose and
contour-certify procedure: determinant roots generate candidate regions, while
local argument-principle counts, contour moments, empirical margins, and
persistence over determinant orders and shifts certify visible poles.  A
Rouch\'e-type perturbation analysis gives sufficient conditions for stable local
zero counts and explains how residues, pole separation, distance to the target
annulus boundary, shifts, and noise affect visibility.  Numerical experiments
verify the pure-pole identity, demonstrate certification under background and
noise, and show that high noise, weak residues, boundary-near poles, and close
poles naturally lead to partial recovery of contour-certified visible poles.}


\maketitle

\section{Introduction}
\label{sec:introduction}

Numerical analytic continuation is a classical ill-conditioned problem.
Although analytic continuation is unique when it exists, its numerical
realization from finite and noisy data is unstable unless suitable a priori
information is imposed.
The severity of this instability depends strongly on both the geometry of the
continuation domain and the assumed function class.
For radial continuation in a disk, a bounded holomorphic prior leads to a
comparatively mild conditional stability mechanism through the Hadamard
three-circles theorem: the loss of accuracy is governed by the logarithmic
radius.
This behavior is substantially different from continuation along strips or
channels, where the loss of accuracy may be much more severe; see, for example,
\cite{Trefethen2020}.

A substantial part of the classical literature treats numerical analytic
continuation as a stabilized reconstruction problem for holomorphic functions,
where stability is obtained from a priori bounds, geometric estimates, rational
approximation, or regularization assumptions.  Cannon and Miller
\cite{Cannon1965} studied stabilized continuation of bounded analytic functions,
while Henrici \cite{Henrici1966} developed a constructive Weierstrass-type
procedure for continuing a holomorphic function from Taylor data along a path.
Approximation-theoretic foundations go back to Walsh's work on boundary values,
Chebyshev approximation, and approximation by analytic or meromorphic functions
\cite{Seidel1942, Walsh1930, Walsh1950}.  Rational and conformal-mapping
approaches provide another important strand: Stefanescu \cite{Stefanescu1980}
studied stable analytic continuation by rational functions, and Bisshopp
\cite{Bisshopp1983} investigated numerical conformal mapping as a tool for
analytic continuation beyond an initially available domain.  In radial
geometries, Franklin \cite{Franklin1990} regularized analytic continuation by
imposing a prescribed bound and gave an FFT-based algorithm with
\(O(m\log m)\) complexity, with error behavior governed by Hadamard's
three-circles principle.  In strip geometries, Fu et al. \cite{Fu2009} proved an
optimal error bound for analytic continuation from approximate boundary-line
data and proposed a modified Tikhonov regularization method attaining the
corresponding order.  More recently, Trefethen \cite{Trefethen2023} advocated
AAA rational approximation as a practical tool for numerical analytic
continuation, especially for extrapolating analytic functions and estimating
complex singularities.  These works provide the classical and modern numerical
background for analytic continuation, but they do not directly address the
local contour certification of exterior poles from determinant characteristics
built from Fourier data.

The present work concerns outward meromorphic continuation from circular
boundary data.  The function is assumed to be holomorphic in the unit disk and
to admit a meromorphic continuation to a larger disk, where exterior poles are
the singularities that obstruct holomorphic continuation.  The stabilized task
is therefore not direct pointwise extrapolation, but the identification and
certification of exterior pole information.  Once such poles are reliably
identified and their principal parts are subtracted, the remaining problem is,
at least formally, reduced to the continuation of a holomorphic residual.  This
viewpoint is close in spirit to Miller's stabilized numerical analytic
prolongation with poles \cite{Miller1970}.  Miller formulated meromorphic
continuation as a stabilized inverse problem and sought a meromorphic
approximant with the minimum number of poles compatible with the data and a
prescribed global bound.  Under separation and nonvanishing-residue assumptions,
he proved H\"older-type stability for the recovered pole locations and residues,
and his SNAPP method produced a posteriori error circles for the computed
poles.  This classical work already identifies pole recovery as a natural
stabilized object in meromorphic continuation.

The algebraic structure exploited in the present paper is related to rational
recovery, exponential analysis, and Hankel-based parameter estimation.  Fourier
coefficients of rational functions contain finite exponential sums, which can
be recovered by Prony-type methods, ESPRIT, matrix pencils, and Hankel
structured techniques.  Derevianko \cite{Derevianko2025} recently developed a
Hankel pencil method for recovering rational functions from Fourier
coefficients, reconstructing poles inside and outside the unit circle through
special Hankel matrix pencils and analyzing the sensitivity of the recovered
poles under structured and unstructured perturbations.  Related work by
Derevianko and coauthors studies rational approximation and exponential-sum
reconstruction from Fourier data, including extended exponential sums, ESPIRA,
and multivariate extensions
\cite{DereviankoHuebner2025,DereviankoPlonka2022,DereviankoPlonkaPetz2021}.

Closely related ideas also appear in noisy analytic continuation and inverse
problems.  Ying \cite{Ying2022,Ying2022_b} used conformal or M\"obius
transformations together with Fourier-domain Prony-type methods to extract pole
or spectral information from limited noisy data.  Recent inverse-problem
studies have likewise used Fourier--Hankel or moment--Hankel structures for
counting and localization, including source counting for heat equations and
topological or phase-center recovery in acoustic inverse scattering
\cite{Deng2026_a,Deng2026_b,Deng2026_c}.  These works show that Fourier or
frequency-domain data, rational structure, and Hankel, Prony, or matrix-pencil
techniques form a powerful framework for parameter recovery.

The present paper differs from these rational, Prony-type, and Hankel-based
recovery methods in both its model and its certification objective.  We study
outward meromorphic continuation from circular boundary data.  The resulting
positive Fourier coefficient sequence contains an unknown holomorphic
background term in addition to the pole-generated exponential sum.  This
background, together with finite boundary sampling and measurement noise,
destroys exact finite-rank structure and may create spurious determinant,
pencil, or Prony roots.  Therefore, it is not reliable to identify the pole
number solely with a numerical Hankel rank, nor to accept all roots produced by
a single algebraic construction.

The main idea of this paper is to replace direct root acceptance by a
contour-counting certification principle.  For each trial order and shift, we
construct a determinant characteristic from the positive Fourier coefficients
of the boundary trace.  In the pure finite-pole case, the determinant for the
correct order factors exactly, and its zeros are the reciprocals of the
exterior poles.  We prove this factorization for both continuous Fourier
coefficients and noiseless equispaced discrete Fourier coefficients.  In the
discrete case, sampling aliasing changes the amplitudes but leaves the
reciprocal pole locations unchanged.  For perturbed data, however, determinant
roots are treated only as candidate reciprocal poles.  Candidate regions are
then tested by local argument-principle counts.  A pole is certified only when a
local contour count remains stable across a family of determinant orders and
shifts.

This visible-pole interpretation is central to the paper.  We do not claim
unconditional recovery of all poles in the exterior annulus.  Poles with small
residues, large modulus, poor separation, or weak determinant margins may be
invisible at the available noise level.  The method therefore reports those
exterior poles that can be certified from the available data, together with
persistence scores and contour margins.  Recovery of the full pole set requires
additional visibility, separation, and noise-margin conditions.

The contributions of this paper are threefold.  First, we derive the
Fourier--determinant structure induced by exterior poles in the disk and prove
exact determinant factorizations for both continuous and noiseless discrete
pure-pole data, including the effect of sampling aliasing.  Second, we develop a
root-propose and contour-certify procedure in which determinant roots generate
candidate regions and local argument-principle counts certify visible poles.
Third, we give a Rouch\'e-type perturbation analysis that connects coefficient
errors, determinant perturbations, contour margins, and stable zero counts; this
analysis explains the observed loss of visibility for weak, close,
boundary-near, or noise-dominated poles.  A residual pole-subtraction check is
included as an additional validation diagnostic.

The paper is organized as follows.
Section~\ref{sec:problem} introduces the meromorphic continuation model and
derives the continuous and discrete Fourier coefficient structures.
Section~\ref{sec:determinant} defines the determinant characteristics and proves
the pure-pole factorization identities.  Section~\ref{sec:argument-principle}
introduces local contour counts, contour moments, the root-propose and
contour-certify procedure, and the residual pole-subtraction diagnostic.
Section~\ref{sec:perturbation-visibility} develops the perturbation,
visibility, and Rouch\'e certification analysis, including the scope of the
visible-pole interpretation.  Section~\ref{sec:numerics} reports numerical
experiments, and Section~\ref{sec:conclusions} concludes the paper.

\section{Problem setting and Fourier coefficient models}
\label{sec:problem}

This section fixes the meromorphic continuation model and derives the
coefficient structures used throughout the paper.  We first formulate the
outward meromorphic continuation problem in the disk and introduce the
reciprocal pole variables.  We then show that the positive Fourier coefficients
of the exact boundary trace coincide with the Taylor coefficients of the
interior holomorphic function.  Finally, we discuss the practical case of
equispaced boundary samples and derive the corresponding aliasing identity.
In the pure-pole case, this discrete aliasing changes only the amplitudes of
the exponential sum and leaves the reciprocal pole locations unchanged.

Let  $\D=\{z\in\C: |z|<1\}$ and for $R>1$, $\D_R=\{z\in\C: |z|<R\}$. Denote by $\Ocal(\D)$ denote the  set  of holomorphic functions on $\D$ and by $\Mcal(\D_R)$  the set of meromorphic functions  on $\D_R$.
We assume that \(f\in\Ocal(\D)\) admits a meromorphic continuation
\(F\in\Mcal(\D_R)\).  The model considered in this paper is
\begin{equation}
    F(z)
    =
    h(z)+\sum_{j=1}^{N}\frac{r_j}{z-p_j},
    \qquad
    h\in\Ocal(\D_{\rho_h}),
    \qquad
    1<|p_j|<R<\rho_h,
    \label{eq:model}
\end{equation}
where the poles \(p_j\) are simple and distinct and \(r_j\neq0\).  The condition
\(\rho_h>R\) means that the holomorphic background has Taylor coefficients that
decay faster than the pole layer to be detected.

The objective is to identify exterior poles that are stably visible from the
data on the unit circle.  We do not assume that all formal poles in
\(\D_R\setminus\overline{\D}\) are recoverable under noise.  Poles with small
residues, large modulus, or poor separation may be invisible at the available
noise level.

We first consider the exact positive Fourier coefficients of the boundary
trace:
\begin{equation}
    a_k
    =
    \frac{1}{2\pi}
    \int_0^{2\pi}
    f(e^{i\theta})e^{-ik\theta}\,\dd\theta,
    \qquad k=0,1,2,\ldots .
    \label{eq:continuous-fourier}
\end{equation}

\begin{lemma}%[Fourier coefficients are Taylor coefficients]
\label{lem:fourier-taylor}
Under the model \eqref{eq:model}, the coefficients \(a_k\) in
\eqref{eq:continuous-fourier} are exactly the Taylor coefficients of \(f\) at
the origin:
\[
    f(z)=\sum_{k=0}^{\infty}a_k z^k,\qquad |z|<1.
\]
\end{lemma}

\begin{proof}
All poles of \(F\) satisfy \(|p_j|>1\), while \(h\) is holomorphic in
\(\D_{\rho_h}\) with \(\rho_h>1\).  Hence \(f=F|_{\D}\) is holomorphic in a
neighbourhood of \(\overline{\D}\).  Applying the change of variables  \(\zeta=e^{i\theta}\), \(\dd\zeta=i e^{i\theta}\dd\theta\),  to Cauchy's coefficient formula
\[
    \frac{f^{(k)}(0)}{k!}
    =
    \frac{1}{2\pi i}
    \int_{|\zeta|=1}
    \frac{f(\zeta)}{\zeta^{k+1}}\,\dd\zeta,
\]
yields
\[
    \frac{f^{(k)}(0)}{k!}
    =
    \frac{1}{2\pi}
    \int_0^{2\pi}
    f(e^{i\theta})e^{-ik\theta}\,\dd\theta
    =
    a_k.
\]
\end{proof}

We define the reciprocal pole variables and the corresponding amplitudes as
\begin{equation}
    \lambda_j=\frac1{p_j},\qquad
    c_j=-\frac{r_j}{p_j}.
    \label{eq:lambda-c}
\end{equation}
Then \(|\lambda_j|<1\), and for \(|z|<1\),
\[
    \frac{r_j}{z-p_j}=-\frac{r_j}{p_j}\frac{1}{1-z/p_j} = c_j\sum_{k=0}^{\infty}\lambda_j^k z^k.
\]
Writing
\[
    h(z)=\sum_{k=0}^{\infty}h_k z^k,
\]
we obtain
\begin{equation}
    a_k = h_k+\sum_{j=1}^{N}c_j\lambda_j^k,
    \qquad k=0,1,2,\ldots .
    \label{eq:coeff-structure}
\end{equation}
Thus exterior poles \(p_j\) in the physical \(z\)-plane correspond to nodes
\(\lambda_j=1/p_j\) inside the unit disk in the auxiliary \(\lambda\)-plane.
The target annulus in the physical plane,
\[
    1<|p|<R,
\]
corresponds to the reciprocal search region
\begin{equation}
    U_R
    =
    \left\{
    \lambda\in\C:\frac1R<|\lambda|<1
    \right\}.
    \label{eq:reciprocal-search-region}
\end{equation}
Thus the exterior poles to be detected correspond to reciprocal nodes
\(\lambda_j\in U_R\).

Throughout the paper, \(p_j\) denotes an exterior pole in the physical
\(z\)-plane, \(r_j\) its residue, \(\lambda_j=1/p_j\) the corresponding
reciprocal pole in the \(\lambda\)-plane, and \(c_j=-r_j/p_j\) the associated
coefficient amplitude.  Exact continuous Fourier coefficients are denoted by
\(a_k\).  Superscripts \(M\) and \((\delta,M)\) indicate, respectively,
quantities computed from noiseless discrete samples and from noisy discrete
samples.  This convention will be used for both Fourier coefficients and the
determinant characteristics introduced below.

We next pass from exact boundary data to equispaced samples.
In computations the boundary values are sampled at equispaced points
\[
    \theta_\ell=\frac{2\pi\ell}{M},
    \qquad \ell=0,\ldots,M-1.
\]
For noiseless samples, define the discrete positive Fourier coefficients
\begin{equation}
    a_k^M
    =
    \frac1M
    \sum_{\ell=0}^{M-1}
    f(e^{i\theta_\ell})e^{-ik\theta_\ell},
    \qquad k=0,\ldots,M-1 .
    \label{eq:discrete-fourier}
\end{equation}
For noisy samples
\[
    y_\ell=f(e^{i\theta_\ell})+\eta_\ell,
\]
we use
\begin{equation}
    a_k^{\delta,M}
    =
    \frac1M
    \sum_{\ell=0}^{M-1}
    y_\ell e^{-ik\theta_\ell}
    =
    a_k^M+\varepsilon_k^M,
    \label{eq:noisy-discrete-fourier}
\end{equation}
where
\[
    \varepsilon_k^M
    =
    \frac1M
    \sum_{\ell=0}^{M-1}
    \eta_\ell e^{-ik\theta_\ell}.
\]
In particular,
\[
    |\varepsilon_k^M|
    \leq
    \max_{0\leq\ell\leq M-1}|\eta_\ell|.
\]

\begin{lemma}%[Aliasing identity]
\label{lem:aliasing}
Assume that \(f\) is holomorphic in a neighbourhood of \(\overline{\D}\) and
has Taylor coefficients \(a_k\).  Then, for \(k=0,\ldots,M-1\),
\begin{equation}
    a_k^M
    =
    \sum_{q=0}^{\infty}a_{k+qM}.
    \label{eq:aliasing}
\end{equation}
\end{lemma}

\begin{proof}
Since \(f\) is holomorphic in a neighbourhood of \(\overline{\D}\), its Taylor
series converges absolutely and uniformly on the unit circle:
\[
    f(e^{i\theta})=\sum_{m=0}^{\infty}a_m e^{im\theta}.
\]
Substituting this expansion into \eqref{eq:discrete-fourier} and using uniform
convergence gives
\[
    a_k^M
    =
    \sum_{m=0}^{\infty}a_m
    \left(
    \frac1M
    \sum_{\ell=0}^{M-1}
    e^{i(m-k)\theta_\ell}
    \right).
\]
The discrete orthogonality relation is
\[
    \frac1M
    \sum_{\ell=0}^{M-1}
    e^{i(m-k)\theta_\ell}
    =
    \begin{cases}
    1, & m\equiv k \pmod M,\\
    0, & m\not\equiv k \pmod M.
    \end{cases}
\]
Since \(m\geq0\) and \(0\leq k\leq M-1\), the admissible indices are
\(m=k+qM\), \(q=0,1,2,\ldots\).  This proves \eqref{eq:aliasing}.
\end{proof}

\begin{lemma}%[Discrete pure-pole coefficients]
\label{lem:discrete-pure-pole}
Suppose
\[
    a_k=\sum_{j=1}^{N}c_j\lambda_j^k,
    \qquad |\lambda_j|<1.
\]
Then the noiseless discrete Fourier coefficients satisfy
\begin{equation}
    a_k^M
    =
    \sum_{j=1}^{N}c_j^M\lambda_j^k,
    \qquad
    c_j^M=\frac{c_j}{1-\lambda_j^M},
    \qquad k=0,\ldots,M-1.
    \label{eq:discrete-pure-pole-coeff}
\end{equation}
\end{lemma}

\begin{proof}
By Lemma~\ref{lem:aliasing},
\[
    a_k^M
    =
    \sum_{q=0}^{\infty}a_{k+qM}
    =
    \sum_{q=0}^{\infty}
    \sum_{j=1}^{N}c_j\lambda_j^{k+qM}.
\]
Since \(|\lambda_j|<1\), the geometric series converges absolutely.  Thus
\[
    a_k^M
    =
    \sum_{j=1}^{N}c_j\lambda_j^k
    \sum_{q=0}^{\infty}\lambda_j^{qM}
    =
    \sum_{j=1}^{N}
    \frac{c_j}{1-\lambda_j^M}\lambda_j^k .
\]
\end{proof}

\begin{remark}%[Aliasing error for holomorphic coefficients]
If \(f\) is holomorphic in \(\D_\rho\) for some \(\rho>1\) and
\(|a_k|\leq C_\rho\rho^{-k}\), then
\[
    |a_k^M-a_k|
    \leq
    \sum_{q=1}^{\infty}|a_{k+qM}|
    \leq
    C_\rho\rho^{-k-M}\frac{1}{1-\rho^{-M}}.
\]
Hence \(a_k^M\to a_k\) exponentially fast as \(M\to\infty\), for fixed \(k\).
\end{remark}

\section{Determinant characteristics}
\label{sec:determinant}

This section introduces the determinant characteristic associated with the
Fourier/Taylor coefficient sequence.  We first define the determinant
characteristic for the exact continuous coefficients and prove its factorization
in the pure-pole case.  We then explain how the same construction is transferred
to noiseless discrete Fourier coefficients and to noisy discrete data.  In the
pure-pole case, equispaced sampling changes only the amplitudes of the
exponential sum and leaves the reciprocal pole locations unchanged.  Therefore
the determinant characteristic has the same zeros in the continuous and
noiseless discrete settings.

For a trial order \(n\geq1\) and a shift \(L\geq0\), define the determinant
characteristic
\begin{equation}
\mathcal D_{n,L}(\lambda)
=
\det
\begin{pmatrix}
a_L & a_{L+1} & \cdots & a_{L+n}\\
a_{L+1} & a_{L+2} & \cdots & a_{L+n+1}\\
\vdots & \vdots & & \vdots\\
a_{L+n-1} & a_{L+n} & \cdots & a_{L+2n-1}\\
1 & \lambda & \cdots & \lambda^n
\end{pmatrix}.
\label{eq:det-characteristic}
\end{equation}
This is a polynomial in \(\lambda\) of degree at most \(n\).  The first \(n\)
rows are formed from the shifted coefficient sequence, while the last row
introduces the variable \(\lambda\).

\begin{proposition}%[Pure-pole determinant identity]
\label{prop:continuous-pure-pole-det}
Suppose that \(h\equiv0\) and
\[
    a_k=\sum_{j=1}^{N}c_j\lambda_j^k,
    \qquad
    c_j\neq0,
    \qquad
    \lambda_i\neq\lambda_j\quad(i\neq j).
\]
Then, for \(n=N\),
\begin{equation}
    \mathcal D_{N,L}(\lambda)
    =
    A_L
    \prod_{j=1}^{N}(\lambda-\lambda_j),
    \label{eq:continuous-det-identity}
\end{equation}
where
\begin{equation}
    A_L
    =
    \left(\prod_{j=1}^{N}c_j\lambda_j^L\right)
    \prod_{1\leq i<j\leq N}(\lambda_j-\lambda_i)^2 .
    \label{eq:AL}
\end{equation}
Consequently, the zeros of \(\mathcal D_{N,L}\) are exactly the reciprocal
exterior poles \(\lambda_j=1/p_j\).
\end{proposition}

\begin{proof}
Set
\[
    v(x)=(1,x,\ldots,x^N).
\]
For \(i=0,\ldots,N-1\), the \(i\)-th row of the upper block in
\eqref{eq:det-characteristic} is
\[
    (a_{L+i},a_{L+i+1},\ldots,a_{L+i+N})
    =
    \sum_{j=1}^{N}c_j\lambda_j^{L+i}v(\lambda_j).
\]
Let
\[
    V=(\lambda_j^i)_{i=0,\ldots,N-1}^{j=1,\ldots,N},
    \qquad
    C=\operatorname{diag}(c_1\lambda_1^L,\ldots,c_N\lambda_N^L),
\]
and let \(W\) be the \(N\times(N+1)\) matrix whose \(j\)-th row is
\(v(\lambda_j)\):
\[
    W=
    \begin{pmatrix}
    v(\lambda_1)\\
    \vdots\\
    v(\lambda_N)
    \end{pmatrix}.
\]
Then the upper \(N\times(N+1)\) block of \eqref{eq:det-characteristic} is
\(VCW\).  Hence
\[
    \mathcal D_{N,L}(\lambda)
    =
    \det
    \begin{pmatrix}
    VCW\\
    v(\lambda)
    \end{pmatrix}.
\]
Since the nodes \(\lambda_j\) are distinct and \(c_j\lambda_j^L\neq0\), the
matrix \(VC\) is nonsingular.  Therefore
\[
    \mathcal D_{N,L}(\lambda)
    =
    \det(VC)
    \det
    \begin{pmatrix}
    W\\
    v(\lambda)
    \end{pmatrix}.
\]
Applying the standard Vandermonde determinant formula to \(V\) and to the
matrix with rows \(v(\lambda_1),\ldots,v(\lambda_N),v(\lambda)\), we obtain
\[
    \mathcal D_{N,L}(\lambda)
    =
    \left(\prod_{j=1}^{N}c_j\lambda_j^L\right)
    \prod_{1\leq i<j\leq N}(\lambda_j-\lambda_i)^2
    \prod_{j=1}^{N}(\lambda-\lambda_j).
\]
This proves the proposition.
\end{proof}

We next define the corresponding determinant characteristic for noiseless
discrete Fourier coefficients.  Let \(a_k^M\) be the coefficients obtained from
\(M\) equispaced boundary samples as in \eqref{eq:discrete-fourier}.  We denote
by
\[
    \mathcal D_{n,L}^{M}(\lambda)
\]
the determinant obtained from \eqref{eq:det-characteristic} by replacing every
coefficient \(a_{L+i+j}\) by \(a_{L+i+j}^{M}\).  We assume
\[
    L+2n\leq M,
\]
so that all coefficients appearing in the determinant are among
\(a_0^M,\ldots,a_{M-1}^M\) and no periodic re-indexing is used.

\begin{proposition}[Discrete pure-pole determinant identity]
\label{prop:discrete-pure-pole-det}
Suppose that
\[
    a_k=\sum_{j=1}^{N}c_j\lambda_j^k,
    \qquad
    |\lambda_j|<1,
    \qquad
    c_j\neq0,
    \qquad
    \lambda_i\neq\lambda_j\quad(i\neq j).
\]
Let \(a_k^M\) be the noiseless discrete Fourier coefficients obtained from
\(M\) equispaced samples.  If \(L+2N\leq M\), then
\begin{equation}
    \mathcal D_{N,L}^{M}(\lambda)
    =
    A_L^M
    \prod_{j=1}^{N}(\lambda-\lambda_j),
    \label{eq:discrete-det-identity}
\end{equation}
where
\begin{equation}
    A_L^M
    =
    \left(
    \prod_{j=1}^{N}
    \frac{c_j}{1-\lambda_j^M}\lambda_j^L
    \right)
    \prod_{1\leq i<j\leq N}(\lambda_j-\lambda_i)^2 .
    \label{eq:ALM}
\end{equation}
Thus the discrete determinant characteristic has the same zeros
\(\lambda_j=1/p_j\) as the continuous determinant characteristic.
\end{proposition}

\begin{proof}
By Lemma~\ref{lem:discrete-pure-pole},
\[
    a_k^M
    =
    \sum_{j=1}^{N}c_j^M\lambda_j^k,
    \qquad
    c_j^M=\frac{c_j}{1-\lambda_j^M}.
\]
Since \(|\lambda_j|<1\), the factors \(1-\lambda_j^M\) are nonzero.  Therefore
the discrete coefficient sequence is again a finite exponential sum with the
same distinct nodes \(\lambda_j\), but with modified nonzero amplitudes
\(c_j^M\).  Applying Proposition~\ref{prop:continuous-pure-pole-det} with
\(c_j\) replaced by \(c_j^M\) gives
\[
    \mathcal D_{N,L}^{M}(\lambda)
    =
    \left(
    \prod_{j=1}^{N}c_j^M\lambda_j^L
    \right)
    \prod_{1\leq i<j\leq N}(\lambda_j-\lambda_i)^2
    \prod_{j=1}^{N}(\lambda-\lambda_j).
\]
Substituting \(c_j^M=c_j/(1-\lambda_j^M)\) gives
\eqref{eq:discrete-det-identity} and \eqref{eq:ALM}.
\end{proof}

\begin{remark}%[Continuous and discrete determinants]
In the pure-pole case, equispaced sampling does not change the determinant
zeros.  It only changes the amplitudes from \(c_j\) to
\[
    c_j^M=\frac{c_j}{1-\lambda_j^M}.
\]
Since \(|\lambda_j|<1\), one has \(c_j^M\to c_j\) exponentially fast as
\(M\to\infty\).
\end{remark}

Finally, for noisy discrete data \(a_k^{\delta,M}\), we denote by
\[
    \mathcal D_{n,L}^{\delta,M}(\lambda)
\]
the determinant obtained from \eqref{eq:det-characteristic} by replacing every
coefficient \(a_{L+i+j}\) by \(a_{L+i+j}^{\delta,M}\).  This determinant is a
perturbation of \(\mathcal D_{n,L}^{M}\).  When a holomorphic background is
present, it is also perturbed by the coefficient sequence \((h_k)\).  The next
section explains how such perturbations are handled by local argument-principle
counts and Rouch\'e-type certification.

%\section{Argument-principle certification of visible poles}
%\label{sec:argument-principle}
%
%This section explains how determinant characteristics are converted into pole
%counts and pole locations.  The key point is that the argument principle is
%first a zero-counting tool for determinant characteristics.  In the pure-pole
%case, these zeros coincide with the reciprocal exterior poles
%\(\lambda_j=1/p_j\).  Thus pole recovery is reformulated as a zero counting and
%zero localization problem in the auxiliary \(\lambda\)-plane.

%%%%%%%%%%%%%%%%%%%%%%%%%%%%%%%%%%%%%%%%%%%%%%%%%%%%%%%%
\section{Argument-principle certification of visible poles}
\label{sec:argument-principle}

This section explains how determinant characteristics are converted into
certified visible pole information.  The argument principle is used first as a
zero-counting tool for determinant characteristics.  In the pure-pole case, the
zeros of the correct determinant characteristic coincide with the reciprocal
exterior poles.  For perturbed data, determinant roots are treated only as
candidates; a pole is accepted only when a local contour count remains stable
over a family of determinant orders and shifts.

The construction proceeds in four steps.  Contour counts are first assigned to
determinant characteristics through the argument principle.  Contour moments are
then used to localize a zero inside a certified local contour.  Next, a
root-propose and contour-certify procedure is introduced to select stable
candidate regions across determinant orders and shifts.  The section concludes
with the certified output and a residual pole-subtraction diagnostic.

\subsection{Argument-principle counts}
\label{subsec:argument-counts}

Let \(D(\lambda)\) denote one of the determinant characteristics introduced in
Section~\ref{sec:determinant}, for example
\[
    D(\lambda)=\mathcal D_{n,L}(\lambda),
    \qquad
    D(\lambda)=\mathcal D_{n,L}^{M}(\lambda),
    \qquad
    D(\lambda)=\mathcal D_{n,L}^{\delta,M}(\lambda).
\]
Let \(\Gamma\) be a positively oriented contour in the \(\lambda\)-plane such
that \(D\) has no zero on \(\Gamma\).  Define
\begin{equation}
    S_\kappa(D,\Gamma)
    :=
    \frac{1}{2\pi i}
    \int_{\Gamma}
    \lambda^\kappa
    \frac{D'(\lambda)}{D(\lambda)}
    \,\dd\lambda,
    \qquad \kappa=0,1,2,\ldots .
    \label{eq:contour-moments}
\end{equation}
The zeroth moment is the logarithmic contour integral
\begin{equation}
    S_0(D,\Gamma)
    =
    \frac{1}{2\pi i}
    \int_{\Gamma}
    \frac{D'(\lambda)}{D(\lambda)}
    \,\dd\lambda .
    \label{eq:S0-count}
\end{equation}
By the argument principle, \(S_0(D,\Gamma)\) is the number of zeros of \(D\)
inside \(\Gamma\), counted with multiplicity.

For the continuous determinant characteristic we write
\begin{equation}
    C_{n,L}(\Gamma)
    =
    S_0(\mathcal D_{n,L},\Gamma),
    \label{eq:continuous-count}
\end{equation}
and for the noisy discrete determinant characteristic we write
\begin{equation}
    C_{n,L}^{\delta,M}(\Gamma)
    =
    S_0(\mathcal D_{n,L}^{\delta,M},\Gamma).
    \label{eq:noisy-discrete-count}
\end{equation}

In the exact pure-pole setting, the interpretation of this count is immediate.
If \(n=N\), then Proposition~\ref{prop:continuous-pure-pole-det} gives
\[
    \mathcal D_{N,L}(\lambda)
    =
    A_L\prod_{j=1}^{N}(\lambda-\lambda_j),
    \qquad A_L\neq0.
\]
Hence the zeros of \(\mathcal D_{N,L}\) are precisely
\(\lambda_1,\ldots,\lambda_N\).  Therefore, for any contour \(\Gamma\) avoiding
the \(\lambda_j\)'s,
\begin{equation}
    C_{N,L}(\Gamma)
    =
    \#\{j:\lambda_j \text{ lies inside }\Gamma\}.
    \label{eq:pure-pole-local-count}
\end{equation}
The same conclusion holds for the noiseless discrete determinant
\(\mathcal D_{N,L}^{M}\), because Proposition~\ref{prop:discrete-pure-pole-det}
shows that equispaced sampling changes only the amplitudes, not the reciprocal
pole locations.

This also explains the role of a large contour in the ideal case.  If the
correct order \(N\) is known and no perturbation is present, a contour enclosing
all reciprocal poles counts them all.  For instance, if
\[
    \max_{1\leq j\leq N}|\lambda_j|<\rho<1,
\]
then
\[
    C_{N,L}(|\lambda|=\rho)=N.
\]
For the target physical annulus \(1<|p|<R\), the corresponding reciprocal
region is
\[
    U_R=\{\lambda\in\C:1/R<|\lambda|<1\}.
\]
In principle, one may count zeros in this annulus by subtracting two circular
counts.  This ideal observation is useful conceptually, but it is not used as a
practical pole-counting rule when the order is unknown and the data are
perturbed.

\subsection{Contour moments and local zero localization}
\label{subsec:local-contour-moments}

The higher contour moments encode the locations of the zeros enclosed by
\(\Gamma\).  The following lemma records the precise relation.

\begin{lemma}
\label{lem:contour-moments}
Let \(D\) be holomorphic in a neighbourhood of the closure of the domain
bounded by \(\Gamma\), and assume that \(D\) has no zero on \(\Gamma\).
Let \(\mu_1,\ldots,\mu_s\) be the distinct zeros of \(D\) inside \(\Gamma\),
with multiplicities \(m_1,\ldots,m_s\).  Then
\begin{equation}
    S_\kappa(D,\Gamma)
    =
    \sum_{\ell=1}^{s}m_\ell\mu_\ell^\kappa,
    \qquad \kappa=0,1,2,\ldots .
    \label{eq:power-sums}
\end{equation}
In particular,
\[
    S_0(D,\Gamma)=\sum_{\ell=1}^{s}m_\ell,
\]
which is the number of zeros inside \(\Gamma\), counted with multiplicity.
\end{lemma}

\begin{proof}
Since \(D\) is holomorphic and has no zero on \(\Gamma\), the logarithmic
derivative \(D'/D\) is meromorphic inside \(\Gamma\).  If \(\mu_\ell\) is a
zero of multiplicity \(m_\ell\), then locally
\[
    D(\lambda)
    =
    (\lambda-\mu_\ell)^{m_\ell}g_\ell(\lambda),
    \qquad
    g_\ell(\mu_\ell)\neq0.
\]
Hence
\[
    \frac{D'(\lambda)}{D(\lambda)}
    =
    \frac{m_\ell}{\lambda-\mu_\ell}
    +
    \frac{g_\ell'(\lambda)}{g_\ell(\lambda)}.
\]
Therefore
\[
    \lambda^\kappa\frac{D'(\lambda)}{D(\lambda)}
\]
has a simple pole at \(\mu_\ell\) with residue \(m_\ell\mu_\ell^\kappa\).
The result follows from the residue theorem.
\end{proof}

When a certified local contour contains exactly one zero, counted with
multiplicity, the zero location is obtained from the first two contour moments.
Indeed, if the unique zero inside \(\Gamma\) is \(\mu\), then
\[
    S_0(D,\Gamma)=1,
    \qquad
    S_1(D,\Gamma)=\mu.
\]
Thus
\begin{equation}
    \mu
    =
    \frac{S_1(D,\Gamma)}{S_0(D,\Gamma)}.
    \label{eq:single-zero-location}
\end{equation}
Consequently, for the empirical determinant characteristic
\(\mathcal D_{n,L}^{\delta,M}\), the reciprocal pole associated with a
single-zero contour is estimated by
\begin{equation}
    \lambda_{n,L}(\Gamma)
    =
    \frac{
    S_1(\mathcal D_{n,L}^{\delta,M},\Gamma)
    }{
    S_0(\mathcal D_{n,L}^{\delta,M},\Gamma)
    },
    \label{eq:local-lambda-estimate}
\end{equation}
and the corresponding physical pole is
\[
    p_{n,L}(\Gamma)=\frac1{\lambda_{n,L}(\Gamma)}.
\]

If a contour encloses more than one zero, the same moments in principle recover
all enclosed zeros.  Let
\[
    q=S_0(D,\Gamma)
\]
be the number of enclosed zeros counted with multiplicity, and write them as
\[
    \nu_1,\ldots,\nu_q,
\]
with repetitions according to multiplicity.  Then
\[
    s_k
    :=
    \sum_{\alpha=1}^{q}\nu_\alpha^k
    =
    S_k(D,\Gamma),
    \qquad k=1,\ldots,q.
\]
The elementary symmetric coefficients \(e_1,\ldots,e_q\) are determined by
Newton's identities,
\[
    k e_k
    =
    \sum_{j=1}^{k}
    (-1)^{j-1}e_{k-j}s_j,
    \qquad
    k=1,\ldots,q,
    \qquad e_0=1.
\]
Thus one obtains the monic polynomial
\[
    P_\Gamma(t)
    =
    t^q-e_1t^{q-1}+e_2t^{q-2}-\cdots+(-1)^q e_q,
\]
whose roots are precisely the zeros inside \(\Gamma\), counted with
multiplicity.

This multi-zero reconstruction is not used as the main numerical rule.  It
requires higher-order moments and a subsequent polynomial root computation,
both of which are sensitive to noise.  The certification procedure below
therefore refines the contour family so that each accepted contour contains one
zero only.  This gives a local and stable estimate through
\eqref{eq:local-lambda-estimate}.

\subsection{Root-propose and contour-certify procedure}
\label{subsec:root-propose-contour-certify}

For perturbed data, a root of a single empirical determinant is not accepted as
a pole.  The empirical determinant
\[
    \mathcal D_{n,L}^{\delta,M}(\lambda)
\]
depends on the trial order \(n\), the shift \(L\), the holomorphic background,
finite sampling, and noise.  Its roots are therefore only candidate reciprocal
poles.  The certification principle is:

\[
    \text{determinant roots propose, local contours certify.}
\]

Let
\[
    \mathcal L=\{L_1,\ldots,L_s\}\subset\mathbb N_0
\]
be a finite set of shifts, and let \(n_{\min}\) and \(n_{\max}\) be the minimum
and maximum trial orders.  We assume that the available coefficients satisfy
\[
    L+2n_{\max}\leq M,
    \qquad L\in\mathcal L .
\]
Define the testing family
\[
    \mathcal I
    =
    \{(n,L): n_{\min}\leq n\leq n_{\max},\ L\in\mathcal L\}.
\]
For each \((n,L)\in\mathcal I\), compute the roots of
\[
    \mathcal D_{n,L}^{\delta,M}(\lambda)=0
\]
inside the reciprocal search region
\[
    U_R
    =
    \left\{\lambda\in\C:\frac1R<|\lambda|<1\right\}.
\]
Collect all such roots in
\[
    \mathcal Z
    =
    \left\{
    z\in U_R:
    \mathcal D_{n,L}^{\delta,M}(z)=0
    \text{ for some }(n,L)\in\mathcal I
    \right\}.
\]
The points in \(\mathcal Z\) are candidate points suggested by different
determinant characteristics.

A genuine visible pole should generate a stable cluster of candidate roots.
By contrast, roots caused by noise, the holomorphic background, or an unsuitable
trial order usually appear less consistently.  We therefore cluster the points
in \(\mathcal Z\).  Let
\[
    \mathcal Z_1,\ldots,\mathcal Z_J
\]
be the resulting clusters, and let \(\zeta_\nu\) be a robust center of
\(\mathcal Z_\nu\), for instance the componentwise median of the real and
imaginary parts.

Around each candidate center \(\zeta_\nu\), choose a positively oriented circle
\[
    \Gamma_\nu
    =
    \{\lambda:|\lambda-\zeta_\nu|=\rho_\nu\}.
\]
The radius \(\rho_\nu\) is chosen so that \(\Gamma_\nu\subset U_R\) and the
contour does not overlap neighbouring candidate clusters.  Whether
\(\Gamma_\nu\) actually certifies a visible pole is decided by local contour
counts.

For each \((n,L)\in\mathcal I\), define
\[
    C_{n,L}^{\delta,M}(\Gamma_\nu)
    =
    \frac{1}{2\pi i}
    \int_{\Gamma_\nu}
    \frac{
    (\mathcal D_{n,L}^{\delta,M})'(\lambda)
    }{
    \mathcal D_{n,L}^{\delta,M}(\lambda)
    }
    \,\dd\lambda .
\]
The hit set of \(\Gamma_\nu\) is
\[
    \mathcal I_\nu
    =
    \left\{
    (n,L)\in\mathcal I:
    C_{n,L}^{\delta,M}(\Gamma_\nu)=1
    \right\}.
\]
Its persistence score is
\begin{equation}
    \operatorname{Pers}(\Gamma_\nu)
    =
    \frac{|\mathcal I_\nu|}{|\mathcal I|}.
    \label{eq:persistence-score}
\end{equation}
This score measures how often the candidate contour encloses exactly one
empirical determinant zero over the selected determinant orders and shifts.

For every \((n,L)\in\mathcal I_\nu\), the enclosed zero is estimated by
\eqref{eq:local-lambda-estimate}.  A stable reciprocal pole estimate associated
with \(\Gamma_\nu\) is then defined by
\begin{equation}
    \widehat\lambda_\nu
    =
    \operatorname{median}
    \left\{
    \lambda_{n,L}(\Gamma_\nu):
    (n,L)\in\mathcal I_\nu
    \right\},
    \label{eq:lambda-hat-median}
\end{equation}
where the median is taken componentwise in the real and imaginary parts.

We also compute the empirical contour margin
\begin{equation}
    \mathfrak m(\Gamma_\nu)
    =
    \operatorname{median}_{(n,L)\in\mathcal I}
    \min_{\lambda\in\Gamma_\nu}
    |\mathcal D_{n,L}^{\delta,M}(\lambda)|.
    \label{eq:empirical-margin}
\end{equation}
A small margin means that the contour passes too close to a zero of an empirical
determinant, in which case both the winding number and the moment estimate may
be unstable.

\subsection{Certified output and residual validation}
\label{subsec:certified-output}
\label{subsec:residual-validation}

We now define the accepted output of the algorithm.

\begin{definition}[Certified visible reciprocal pole]
\label{def:visible-pole}
A candidate contour \(\Gamma_\nu\) certifies one visible reciprocal pole if
\[
    \operatorname{Pers}(\Gamma_\nu)\geq\tau_{\rm pers},
    \qquad
    \mathfrak m(\Gamma_\nu)\geq\tau_{\rm cont},
\]
and if \(\Gamma_\nu\) is disjoint from the previously accepted contours.  The
certified reciprocal pole is \(\widehat\lambda_\nu\), and the corresponding
physical pole is
\[
    \widehat p_\nu=\frac1{\widehat\lambda_\nu}.
\]
\end{definition}

The estimated number of visible poles is
\begin{equation}
    \widehat N_{\rm vis}
    =
    \#\{\nu:\Gamma_\nu
    \text{ certifies one visible reciprocal pole}\}.
    \label{eq:Nvisible}
\end{equation}
This number is not the total number of all formal exterior poles in the
meromorphic continuation model.  It is the number of poles in the target annulus
\(1<|p|<R\) that are visible through stable determinant-root clustering and
local contour-count certification.  The empirical margin
\(\mathfrak m(\Gamma_\nu)\) is used here as a numerical safeguard; its
theoretical role is explained by the Rouch\'e analysis in
Section~\ref{sec:perturbation-visibility}.

After the visible reciprocal poles have been certified, their coefficient
amplitudes can be estimated from the same Fourier coefficient sequence.  For an
index set \(K_{\rm fit}\subset\{0,\ldots,M-1\}\), compute
\[
    (\widehat c_1,\ldots,\widehat c_{\widehat N_{\rm vis}})
    =
    \arg\min_{d_1,\ldots,d_{\widehat N_{\rm vis}}}
    \sum_{k\in K_{\rm fit}}
    \left|
    a_k^{\delta,M}
    -
    \sum_{\nu=1}^{\widehat N_{\rm vis}}
    d_\nu \widehat\lambda_\nu^k
    \right|^2 .
\]
The residual coefficients are then defined by
\begin{equation}
    \widehat h_k
    =
    a_k^{\delta,M}
    -
    \sum_{\nu=1}^{\widehat N_{\rm vis}}
    \widehat c_\nu\widehat\lambda_\nu^k .
    \label{eq:residual-coefficients}
\end{equation}

This subtraction step is a post-certification diagnostic, not the primary
pole-counting rule.  If the certified poles represent genuine dominant
singularities, the residual sequence \((\widehat h_k)\) should decay faster
than the original sequence \((a_k^{\delta,M})\), consistent with a holomorphic
background whose nearest remaining singularity lies farther away.  Failure of
faster residual decay does not invalidate a contour certification by itself,
but it indicates that unresolved poles, poorly separated clusters, or
background terms may still contribute over the available coefficient range.

\section{Perturbation, visibility, and contour certification}
\label{sec:perturbation-visibility}

The exact factorization in the pure-pole model explains why determinant roots
are natural candidate reciprocal poles.  In practice, however, the available
coefficients contain contributions from the holomorphic background, finite
sampling, and noise.  We now give a perturbation criterion under which a local
contour count remains stable.

For continuous exact coefficients, the pure-pole reference sequence is
\[
    a_k^{\rm p}
    =
    \sum_{j=1}^{N}c_j\lambda_j^k.
\]
For noiseless equispaced discrete coefficients, the corresponding pure-pole
reference sequence is
\[
    a_k^{{\rm p},M}
    =
    \sum_{j=1}^{N}c_j^M\lambda_j^k,
    \qquad
    c_j^M=\frac{c_j}{1-\lambda_j^M}.
\]
Both sequences have the same nodes \(\lambda_j\).  In this section
\(\mathcal D_{n,L}^{\rm ref}\) denotes the pure-pole reference determinant,
either continuous or noiseless discrete, depending on the data model.  Similarly,
\(a_k^{\rm ref}\) denotes the corresponding reference coefficient sequence.  The
perturbed sequence is written as
\[
    \widetilde a_k
    =
    a_k^{\rm ref}+b_k,
\]
where \(b_k\) contains the holomorphic background, aliasing from the background,
and measurement noise.  The perturbed determinant constructed from
\(\widetilde a_k\) is denoted by
\(\widetilde{\mathcal D}_{n,L}\).  In the numerical algorithm this object is the
empirical determinant \(\mathcal D_{n,L}^{\delta,M}\).

For fixed \(n\) and \(L\), set
\begin{equation}
    \omega_{n,L}
    =
    \max_{0\leq q\leq 2n-1}|b_{L+q}|.
    \label{eq:coefficient-perturbation-size}
\end{equation}
This quantity measures the coefficient perturbation over exactly the coefficient
segment used in the determinant characteristic.

\begin{lemma}[Coefficient perturbation from a holomorphic background]
\label{lem:coefficient-perturbation}
Assume that the holomorphic background satisfies
\[
    h(z)=\sum_{k=0}^{\infty}h_k z^k,
    \qquad
    |h_k|\leq C_h\rho_h^{-k},
    \qquad \rho_h>R.
\]
For continuous Fourier coefficients with coefficient noise \(e_k\), one has
\[
    b_k=h_k+e_k.
\]
Consequently,
\begin{equation}
    \omega_{n,L}
    \leq
    C_h\rho_h^{-L}
    +
    \max_{0\leq q\leq 2n-1}|e_{L+q}|.
    \label{eq:continuous-omega-bound}
\end{equation}
For noiseless equispaced discrete Fourier coefficients with \(M\) samples, the
background aliasing contribution satisfies
\begin{equation}
    \left|
    \sum_{q=0}^{\infty}h_{k+qM}
    \right|
    \leq
    \frac{C_h\rho_h^{-k}}{1-\rho_h^{-M}},
    \qquad 0\leq k\leq M-1.
    \label{eq:background-aliasing-bound}
\end{equation}
Hence the discrete coefficient perturbation is bounded by
\begin{equation}
    \omega_{n,L}^{M}
    \leq
    \frac{C_h\rho_h^{-L}}{1-\rho_h^{-M}}
    +
    \max_{0\leq q\leq 2n-1}|e_{L+q}^{M}|.
    \label{eq:discrete-omega-bound}
\end{equation}
\end{lemma}

\begin{proof}
The continuous estimate follows directly from \(b_k=h_k+e_k\) and the assumed
coefficient decay of \(h\).  For the discrete estimate, the aliasing identity gives
\[
    h_k^M
    =
    \sum_{q=0}^{\infty}h_{k+qM}.
\]
Using \(|h_k|\leq C_h\rho_h^{-k}\), we obtain
\[
    |h_k^M|
    \leq
    C_h\rho_h^{-k}
    \sum_{q=0}^{\infty}\rho_h^{-qM}
    =
    \frac{C_h\rho_h^{-k}}{1-\rho_h^{-M}}.
\]
Taking the maximum over \(k=L,\ldots,L+2n-1\) gives
\eqref{eq:discrete-omega-bound}.
\end{proof}

The shift \(L\) therefore has a precise role: it suppresses the holomorphic
background by the factor \(\rho_h^{-L}\).  However, large shifts also reduce the
pole signal through factors \(\lambda_j^L\).  This trade-off is one reason why
the method tests several shifts rather than using a single one.

\begin{lemma}[Determinant perturbation bound]
\label{lem:determinant-perturbation}
Let \(\Gamma\) be a compact contour in the \(\lambda\)-plane.  Suppose that all
entries of the unperturbed determinant matrix defining
\(\mathcal D_{n,L}^{\rm ref}\) are bounded by \(B_\Gamma\) on \(\Gamma\), and
that the coefficient perturbations satisfy \eqref{eq:coefficient-perturbation-size}.
Then
\begin{equation}
    \max_{\lambda\in\Gamma}
    \left|
    \widetilde{\mathcal D}_{n,L}(\lambda)
    -
    \mathcal D_{n,L}^{\rm ref}(\lambda)
    \right|
    \leq
    (n+1)!
    \left[
    (B_\Gamma+\omega_{n,L})^{n+1}
    -
    B_\Gamma^{n+1}
    \right].
    \label{eq:determinant-perturbation-bound}
\end{equation}
In particular, for small \(\omega_{n,L}\),
\begin{equation}
    \max_{\lambda\in\Gamma}
    \left|
    \widetilde{\mathcal D}_{n,L}(\lambda)
    -
    \mathcal D_{n,L}^{\rm ref}(\lambda)
    \right|
    \leq
    (n+1)!(n+1)(B_\Gamma+\omega_{n,L})^{n}\omega_{n,L}.
    \label{eq:linear-det-bound}
\end{equation}
\end{lemma}

\begin{proof}
The determinant is a finite sum of products of \(n+1\) matrix entries.  By the
Leibniz formula,
\[
    \det A
    =
    \sum_{\sigma\in S_{n+1}}
    \operatorname{sgn}(\sigma)
    \prod_{\alpha=0}^{n}A_{\alpha,\sigma(\alpha)}.
\]
For each permutation, the difference between the perturbed and unperturbed
products is bounded by
\[
    (B_\Gamma+\omega_{n,L})^{n+1}-B_\Gamma^{n+1}.
\]
Summing over the \((n+1)!\) permutations gives
\eqref{eq:determinant-perturbation-bound}.  The linearized bound
\eqref{eq:linear-det-bound} follows from the mean value theorem applied to
\(x^{n+1}\).
\end{proof}

\begin{lemma}[Pure-pole contour margin]
\label{lem:pure-pole-margin}
Assume the pure-pole model and take \(n=N\).  Let
\[
    \Gamma_j(\rho)
    =
    \{\lambda:|\lambda-\lambda_j|=\rho\},
\]
where
\[
    0<\rho<
    \min\left\{
    \operatorname{dist}(\lambda_j,\partial U_R),
    \frac12\min_{i\neq j}|\lambda_i-\lambda_j|
    \right\}.
\]
Then \(\Gamma_j(\rho)\) encloses \(\lambda_j\) and no other reciprocal pole.
Moreover,
\begin{equation}
    \min_{\lambda\in\Gamma_j(\rho)}
    |\mathcal D_{N,L}^{\rm ref}(\lambda)|
    \geq
    |A_L^{\rm ref}|\,
    \rho
    \prod_{i\neq j}
    \bigl(|\lambda_i-\lambda_j|-\rho\bigr),
    \label{eq:pure-pole-margin}
\end{equation}
where
\[
    \mathcal D_{N,L}^{\rm ref}(\lambda)
    =
    A_L^{\rm ref}\prod_{i=1}^{N}(\lambda-\lambda_i).
\]
Here \(A_L^{\rm ref}=A_L\) for continuous coefficients and
\(A_L^{\rm ref}=A_L^M\) for noiseless discrete coefficients.
\end{lemma}

\begin{proof}
On \(\Gamma_j(\rho)\), \(|\lambda-\lambda_j|=\rho\).  For \(i\neq j\), the
triangle inequality gives
\[
    |\lambda-\lambda_i|
    \geq
    |\lambda_i-\lambda_j|-|\lambda-\lambda_j|
    =
    |\lambda_i-\lambda_j|-\rho.
\]
Multiplying these lower bounds in the pure-pole factorization gives
\eqref{eq:pure-pole-margin}.
\end{proof}

Lemma~\ref{lem:pure-pole-margin} makes the visibility mechanism explicit.  The
local margin decreases when residues are small, when poles are close, when a
pole approaches the boundary of the search annulus, or when the shifted pole
signal becomes weak.

Combining the determinant perturbation bound with Rouch\'e's theorem gives the
main local certification result.

\begin{theorem}[Rouch\'e certification of a visible reciprocal pole]
\label{thm:visible-pole-certification}
Assume the pure-pole reference determinant has distinct reciprocal poles
\(\lambda_1,\ldots,\lambda_N\), and let \(\Gamma_j(\rho)\) be as in
Lemma~\ref{lem:pure-pole-margin}.  Suppose the perturbed determinant
\(\widetilde{\mathcal D}_{N,L}\) satisfies
\begin{equation}
    \Delta_{N,L}(\Gamma_j)
    :=
    \max_{\lambda\in\Gamma_j(\rho)}
    \left|
    \widetilde{\mathcal D}_{N,L}(\lambda)
    -
    \mathcal D_{N,L}^{\rm ref}(\lambda)
    \right|
    <
    \min_{\lambda\in\Gamma_j(\rho)}
    |\mathcal D_{N,L}^{\rm ref}(\lambda)|.
    \label{eq:rouche-visible-condition}
\end{equation}
Then \(\widetilde{\mathcal D}_{N,L}\) has exactly one zero inside
\(\Gamma_j(\rho)\).  Equivalently,
\begin{equation}
    \frac{1}{2\pi i}
    \int_{\Gamma_j(\rho)}
    \frac{
    \widetilde{\mathcal D}_{N,L}'(\lambda)
    }{
    \widetilde{\mathcal D}_{N,L}(\lambda)
    }
    \,\dd\lambda
    =
    1.
    \label{eq:certified-single-count}
\end{equation}
\end{theorem}

\begin{proof}
By the pure-pole factorization,
\[
    \mathcal D_{N,L}^{\rm ref}(\lambda)
    =
    A_L^{\rm ref}\prod_{i=1}^{N}(\lambda-\lambda_i),
\]
and the contour \(\Gamma_j(\rho)\) encloses exactly one zero, namely
\(\lambda_j\).  Condition \eqref{eq:rouche-visible-condition} is precisely the
hypothesis of Rouch\'e's theorem for the pair
\[
    \mathcal D_{N,L}^{\rm ref}
    \quad\text{and}\quad
    \widetilde{\mathcal D}_{N,L}-\mathcal D_{N,L}^{\rm ref}.
\]
Hence \(\mathcal D_{N,L}^{\rm ref}\) and \(\widetilde{\mathcal D}_{N,L}\) have
the same number of zeros inside \(\Gamma_j(\rho)\), counted with multiplicity.
The contour integral identity then follows from the argument principle.
\end{proof}

The theorem motivates the empirical margin used in the algorithm.  The reference
margin
\[
    \min_{\lambda\in\Gamma}|\mathcal D_{N,L}^{\rm ref}(\lambda)|
\]
is not known in computations.  Instead, we use the observable quantity
\[
    \min_{\lambda\in\Gamma}|\mathcal D_{n,L}^{\delta,M}(\lambda)|
\]
as a numerical safeguard.  A small empirical margin indicates that the contour
passes close to a determinant zero, so the winding number and the moment estimate
may be unstable.

The preceding theorem is stated for the correct order \(n=N\).  In practice,
\(N\) is unknown and several pairs \((n,L)\) are tested.  For a candidate contour
\(\Gamma\), define
\[
    \mathcal I(\Gamma)
    =
    \left\{
    (n,L)\in\mathcal I:
    C_{n,L}^{\delta,M}(\Gamma)=1
    \right\}.
\]
The persistence score is
\begin{equation}
    \operatorname{Pers}(\Gamma)
    =
    \frac{|\mathcal I(\Gamma)|}{|\mathcal I|}.
    \label{eq:pers-score-theory}
\end{equation}
For a candidate contour \(\Gamma_\nu\), we also write
\(\mathcal I_\nu:=\mathcal I(\Gamma_\nu)\).

\begin{definition}[Visible reciprocal pole]
\label{def:visible-reciprocal-pole}
A reciprocal pole \(\lambda_j\in U_R\) is called visible relative to the testing
family \(\mathcal I\), the contour family, and the noise level if there exists a
candidate contour \(\Gamma\subset U_R\) such that
\begin{enumerate}[label=(\roman*)]
\item \(\Gamma\) encloses \(\lambda_j\) and no other reciprocal pole;
\item the Rouch\'e condition \eqref{eq:rouche-visible-condition} holds for a
positive fraction of pairs \((n,L)\in\mathcal I\);
\item the empirical margin on \(\Gamma\) is above the prescribed threshold.
\end{enumerate}
\end{definition}

This definition is intentionally relative to the data quality and to the testing
family.  A pole with a very small residue, a pole close to another pole, or a pole
near the boundary of \(U_R\) may fail to be visible at a given noise level, even
though it is present in the meromorphic model.  When a candidate contour is
certified to contain one zero, the zero location is recovered by the moment formula
\eqref{eq:single-zero-location}.

\begin{remark}[Visible-pole interpretation and scope]
\label{rem:visible-pole-scope}
The determinant-characteristic method should be interpreted as a visible-pole
certification procedure, not as an unconditional recovery method for every pole
in \(\D_R\setminus\overline{\D}\).  A pole may fail to be certified when its
residue is small relative to the noise level, when it is far from the unit
circle so that the contribution \(c_j\lambda_j^k\) decays rapidly, when nearby
reciprocal poles lead to ill-conditioned Vandermonde factors, or when the
holomorphic background masks the pole contribution over the available range of
Fourier indices.  Equivalently, in the perturbation analysis, loss of visibility
occurs when the determinant perturbation on the relevant contour is comparable
to the pure-pole contour margin.

Thus the natural output of the method is the set of contour-certified visible
poles, together with persistence scores and contour margins.  Full recovery of
all exterior poles requires additional visibility, separation, and noise-margin
conditions.
\end{remark}

\section{Numerical experiments}
\label{sec:numerics}

We present numerical experiments to illustrate the proposed certification
principle rather than to provide an exhaustive benchmark comparison.  The
experiments focus on four questions: whether the pure-pole determinant
factorization is observed numerically, whether visible poles can be certified
in the presence of a holomorphic background and noise, how localization and
persistence change with noise, and how small residues, boundary proximity, and
poor separation lead to loss of visibility.

All experiments are performed in the target annulus
\[
    1<|p|<R,\qquad R=2.2,
\]
which corresponds to the reciprocal search region
\[
    U_R=\{\lambda\in\mathbb C:1/R<|\lambda|<1\}.
\]
Unless otherwise stated, the trial determinant orders and shifts are
\[
    n=2,\ldots,6,\qquad
    L\in\{4,8,12,16,20,24\}.
\]
The method reports contour-certified visible poles rather than all formal
poles in the meromorphic continuation model.

The experiments should be interpreted in light of
Theorem~\ref{thm:visible-pole-certification}.  The theorem gives a sufficient
condition for local zero-count stability: the determinant perturbation on a
contour must be smaller than the pure-pole contour margin.  In computations,
this exact margin is unknown.  Instead, we observe its numerical consequences:
persistent one-zero contour counts, stable moment-based pole estimates, and
large empirical contour margins.  When residues are small, poles are close,
poles lie near the boundary of \(U_R\), or the noise level increases, these
stability indicators deteriorate and some formal poles may cease to be
certified.

For localization accuracy, when the estimated and true pole numbers agree, we
use the maximum matching error
\[
    e_{\max}
    =
    \min_{\pi\in\mathfrak S_N}
    \max_{1\leq j\leq N}
    |\widehat\lambda_{\pi(j)}-\lambda_j|,
\]
where \(\mathfrak S_N\) denotes the set of permutations of
\(\{1,\ldots,N\}\).  When the estimated and true pole numbers differ, we report
nearest-neighbour errors from each true pole to the closest certified estimate.

\subsection{Pure-pole determinant identity}
\label{subsec:pure-pole-numerics}

We first test the exact pure-pole model without holomorphic background and
without noise.  The exterior poles are
\[
\begin{aligned}
    p_1 &= 1.25,\\
    p_2 &= 1.45e^{0.45i},\\
    p_3 &= 1.70e^{-0.55i}.
\end{aligned}
\]
The corresponding reciprocal poles are
\[
\lambda_1=0.800000,\qquad
\lambda_2=0.620998-0.299976i,\qquad
\lambda_3=0.501485+0.307463i.
\]
For the correct order \(n=N=3\), the determinant roots agree with the true
reciprocal poles up to round-off error.  Table~\ref{tab:pure-pole} reports the
maximum matching error for several shifts.

\begin{table}[htbp]
\centering
\caption{Pure-pole determinant identity.  For \(n=N=3\), the roots of
\(\mathcal D_{N,L}\) agree with the reciprocal poles for all tested shifts.}
\label{tab:pure-pole}
\begin{tabular}{ccc}
\toprule
Shift \(L\) & Number of roots in \(|\lambda|<0.95\) & Maximum matching error \\
\midrule
0  & 3 & \(4.442\times10^{-14}\) \\
4  & 3 & \(1.225\times10^{-13}\) \\
8  & 3 & \(1.013\times10^{-13}\) \\
12 & 3 & \(5.040\times10^{-13}\) \\
16 & 3 & \(3.420\times10^{-12}\) \\
20 & 3 & \(1.023\times10^{-11}\) \\
\bottomrule
\end{tabular}
\end{table}

\begin{figure}[htbp]
\centering
\includegraphics[width=0.66\textwidth]{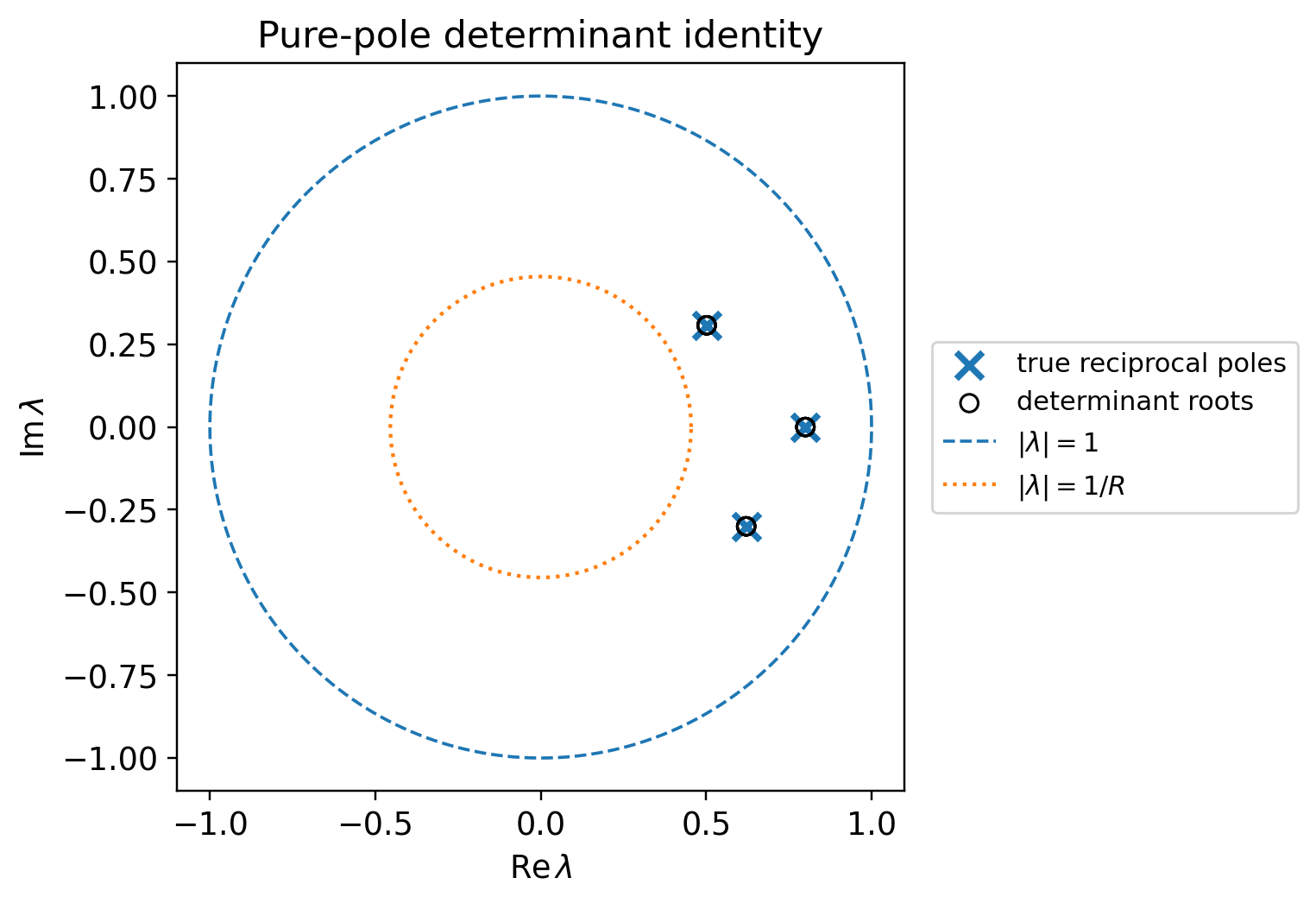}
\caption{Pure-pole determinant identity.  In the exact pure-pole case, the
determinant roots coincide with the true reciprocal poles.}
\label{fig:pure-pole}
\end{figure}

This experiment verifies the reference situation used in the Rouch\'e
comparison: in the absence of background, aliasing error, and noise, the
determinant characteristic has the exact pure-pole factorization.

\subsection{Visible-pole certification with background and noise}
\label{subsec:three-pole}

We next use the same three poles, now with a holomorphic background and
coefficient noise of level \(10^{-10}\).  The raw determinant roots are not
accepted directly as poles.  They are used only to propose candidate regions,
which are then tested by local argument-principle counts.

The algorithm produces 92 raw candidate roots in the target annulus, forming 18
candidate clusters.  After contour-count certification, three visible poles are
accepted:
\[
\begin{aligned}
\widehat\lambda_1 &= 0.800000+0.000000i,\\
\widehat\lambda_2 &= 0.620999-0.299976i,\\
\widehat\lambda_3 &= 0.501491+0.307463i.
\end{aligned}
\]
The corresponding persistence scores are
\[
    0.900,\qquad 0.833,\qquad 0.767.
\]
The maximum matching error is
\[
    e_{\max}=5.671\times10^{-6},
\]
and the median nearest-neighbour error is
\[
    7.919\times10^{-7}.
\]

\begin{figure}[htbp]
\centering
\includegraphics[width=0.70\textwidth]{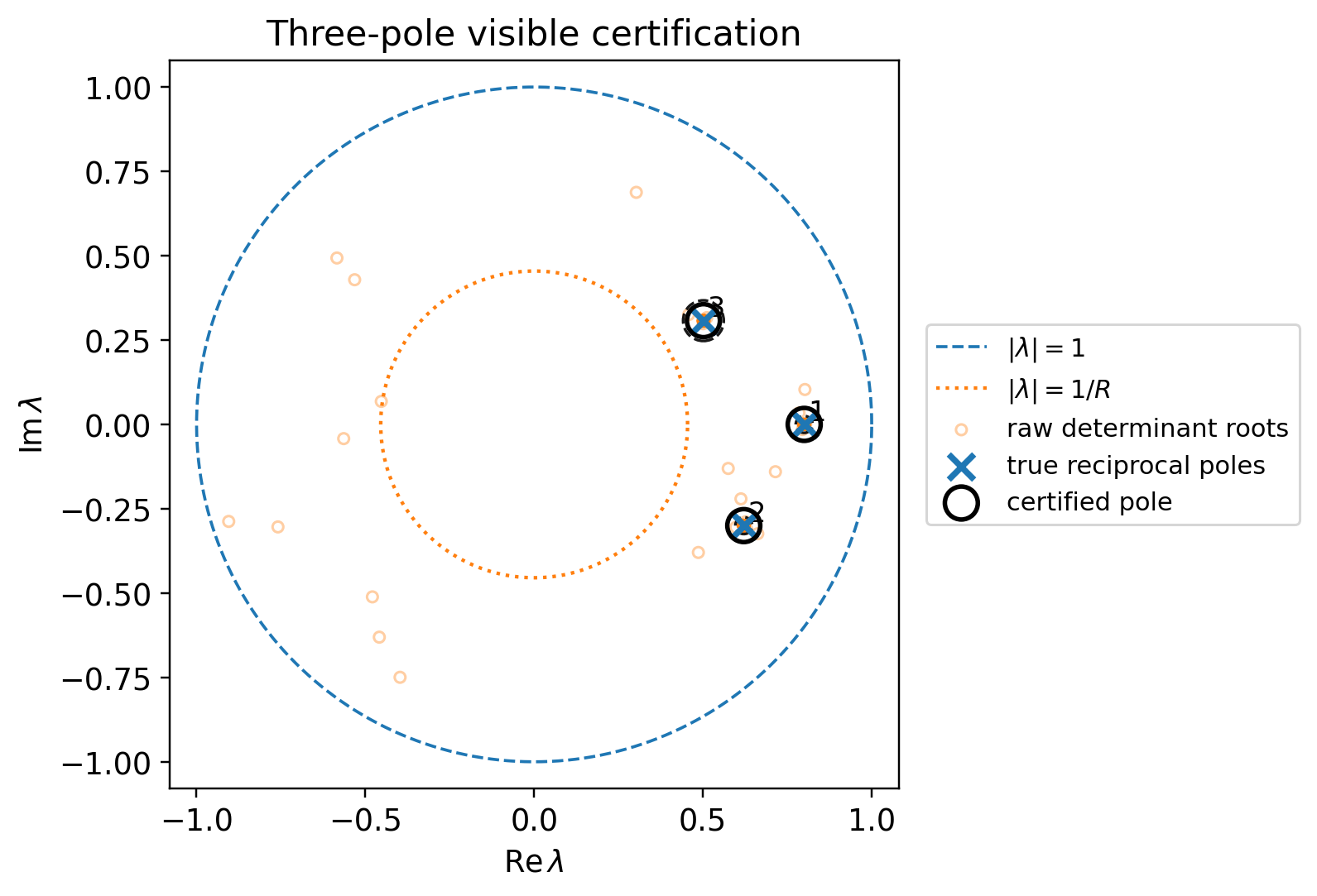}
\caption{Visible-pole certification with holomorphic background and noise.
Raw determinant roots are used only to propose candidate regions.  The certified
visible reciprocal poles are obtained by local contour-count persistence.}
\label{fig:three-pole}
\end{figure}

Figure~\ref{fig:three-pole} shows that many raw determinant roots do not pass
the certification step.  The reliable objects are not isolated roots of a
single determinant characteristic, but clusters whose local zero count remains
stable across determinant orders and shifts.  This supports the
root-propose and contour-certify interpretation of the method.

\subsection{Noise robustness}
\label{subsec:noise}

We test the same three-pole model under increasing coefficient noise.
Table~\ref{tab:noise} reports the certified pole number, localization error,
and persistence scores.  The algorithm certifies three visible poles for all
tested noise levels from \(0\) to \(10^{-6}\), but the localization error grows
and the weakest persistence score decreases as the noise level increases.

\begin{table}[htbp]
\centering
\caption{Noise robustness for the three-pole model.}
\label{tab:noise}
\begin{tabular}{ccccc}
\toprule
Noise level & \(\widehat N_{\rm vis}\) & \(e_{\max}\) &
Median nearest error & Persistence scores \\
\midrule
\(0\)        & 3 & \(3.103\times10^{-8}\) & \(1.297\times10^{-8}\) & \(0.77,\ 0.70,\ 0.60\) \\
\(10^{-12}\) & 3 & \(5.341\times10^{-7}\) & \(1.499\times10^{-7}\) & \(0.90,\ 0.83,\ 0.77\) \\
\(10^{-10}\) & 3 & \(2.680\times10^{-5}\) & \(1.438\times10^{-6}\) & \(0.90,\ 0.83,\ 0.77\) \\
\(10^{-8}\)  & 3 & \(1.807\times10^{-4}\) & \(4.352\times10^{-5}\) & \(0.90,\ 0.83,\ 0.57\) \\
\(10^{-6}\)  & 3 & \(3.685\times10^{-3}\) & \(5.456\times10^{-4}\) & \(0.90,\ 0.60,\ 0.30\) \\
\bottomrule
\end{tabular}
\end{table}

This behavior is consistent with Lemma~\ref{lem:determinant-perturbation} and
Theorem~\ref{thm:visible-pole-certification}.  Increasing the noise level
increases the effective coefficient perturbation size \(\omega_{n,L}\), and
therefore increases the determinant perturbation on the contour.  The
high-noise case \(10^{-6}\) is particularly informative: although all three
poles are still certified, the third persistence score drops to \(0.30\),
indicating that its local margin is close to the perturbation scale.

\subsection{Visibility limitations}
\label{subsec:limitations-numerics}

Finally, we examine three difficult configurations that explain the
visible-pole formulation.  These tests probe the main factors in the pure-pole
margin lower bound \eqref{eq:pure-pole-margin}: the amplitude factor through
\(A_L^{\rm ref}\), the separation factors
\(|\lambda_i-\lambda_j|-\rho\), and the admissible contour radius constrained by
\(\operatorname{dist}(\lambda_j,\partial U_R)\).

In the first case, one pole has a small residue,
\[
    |r_2|=6.325\times10^{-3},
\]
with noise level \(10^{-8}\).  The method still certifies all three poles, but
the weak-residue pole has a lower persistence score, \(0.444\), compared with
\(1.000\) and \(0.611\) for the other two certified poles.  Thus a small residue
does not necessarily make a pole invisible, but it reduces the persistence
evidence.

In the second case, the third pole satisfies
\[
    |p_3|=2.15,\qquad |\lambda_3|=0.4651.
\]
Since \(1/R\approx0.4545\), this reciprocal pole lies close to the inner
boundary of the target annulus.  The method certifies only two visible poles.
This agrees with the margin estimate: when \(\lambda_j\) approaches
\(\partial U_R\), the admissible contour radius must shrink, and the
corresponding contour margin decreases.

In the third case, two reciprocal poles are close:
\[
    \lambda_2=0.725975-0.147162i,
    \qquad
    \lambda_3=0.710705-0.166407i,
\]
with spacing
\[
    |\lambda_2-\lambda_3|=2.457\times10^{-2}.
\]
Only one visible pole is certified.  Closely spaced poles may fail to generate
separated certified clusters under the chosen contour and clustering scale,
which reflects the separation product in \eqref{eq:pure-pole-margin}.

\begin{figure}[htbp]
\centering
\begin{minipage}{0.32\textwidth}
\centering
\includegraphics[width=\textwidth]{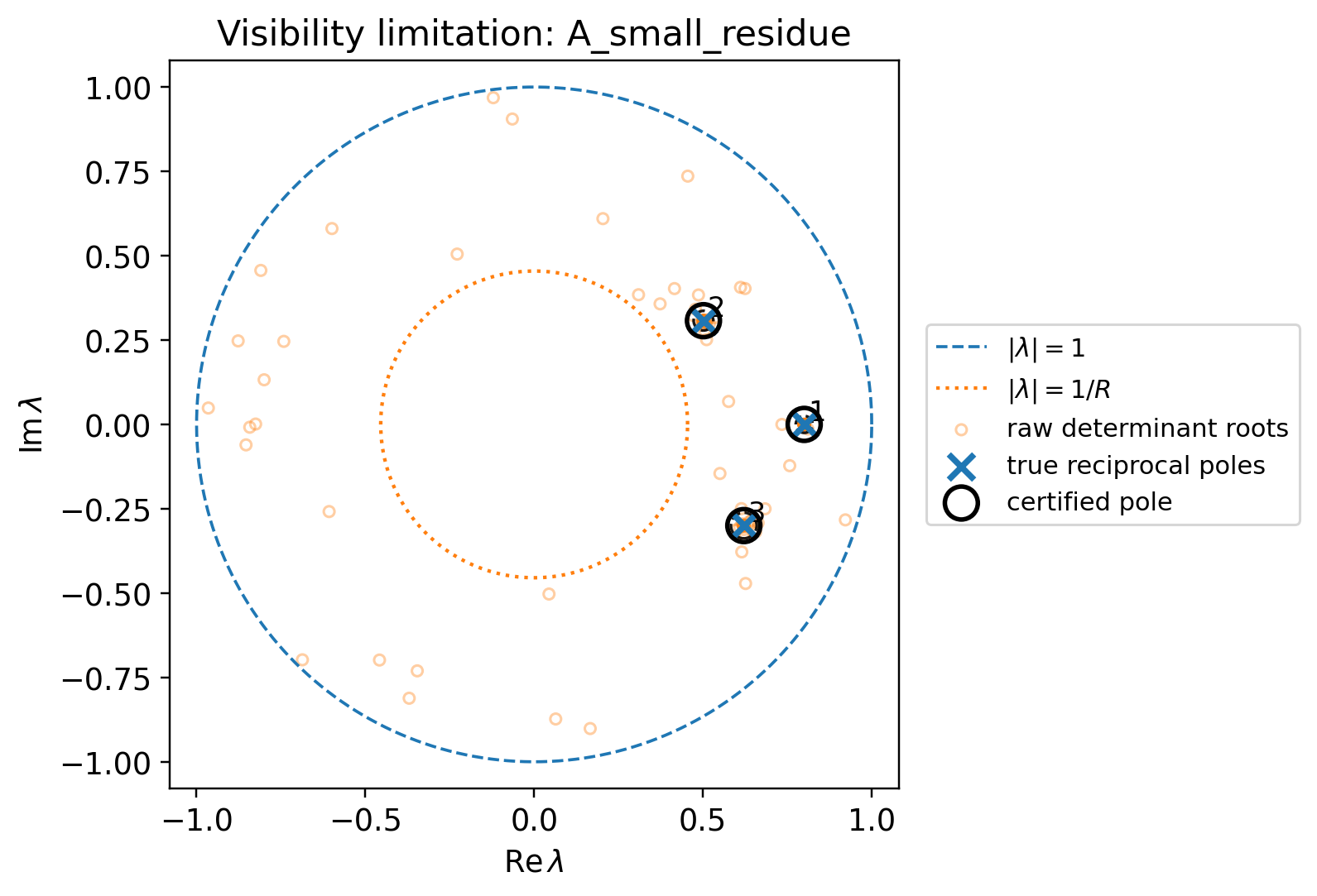}\\
{\small (a) Small residue}
\end{minipage}
\hfill
\begin{minipage}{0.32\textwidth}
\centering
\includegraphics[width=\textwidth]{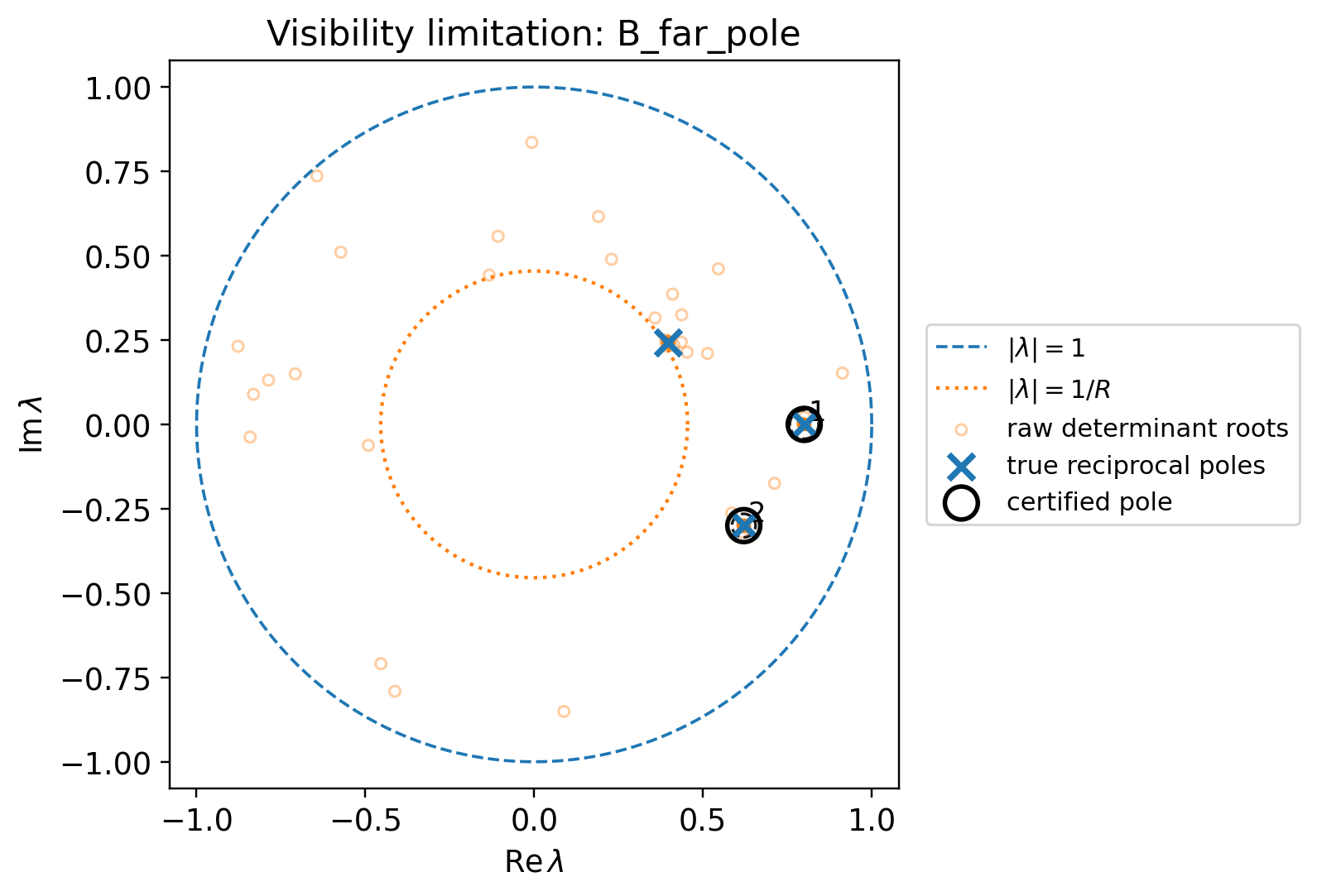}\\
{\small (b) Boundary-near pole}
\end{minipage}
\hfill
\begin{minipage}{0.32\textwidth}
\centering
\includegraphics[width=\textwidth]{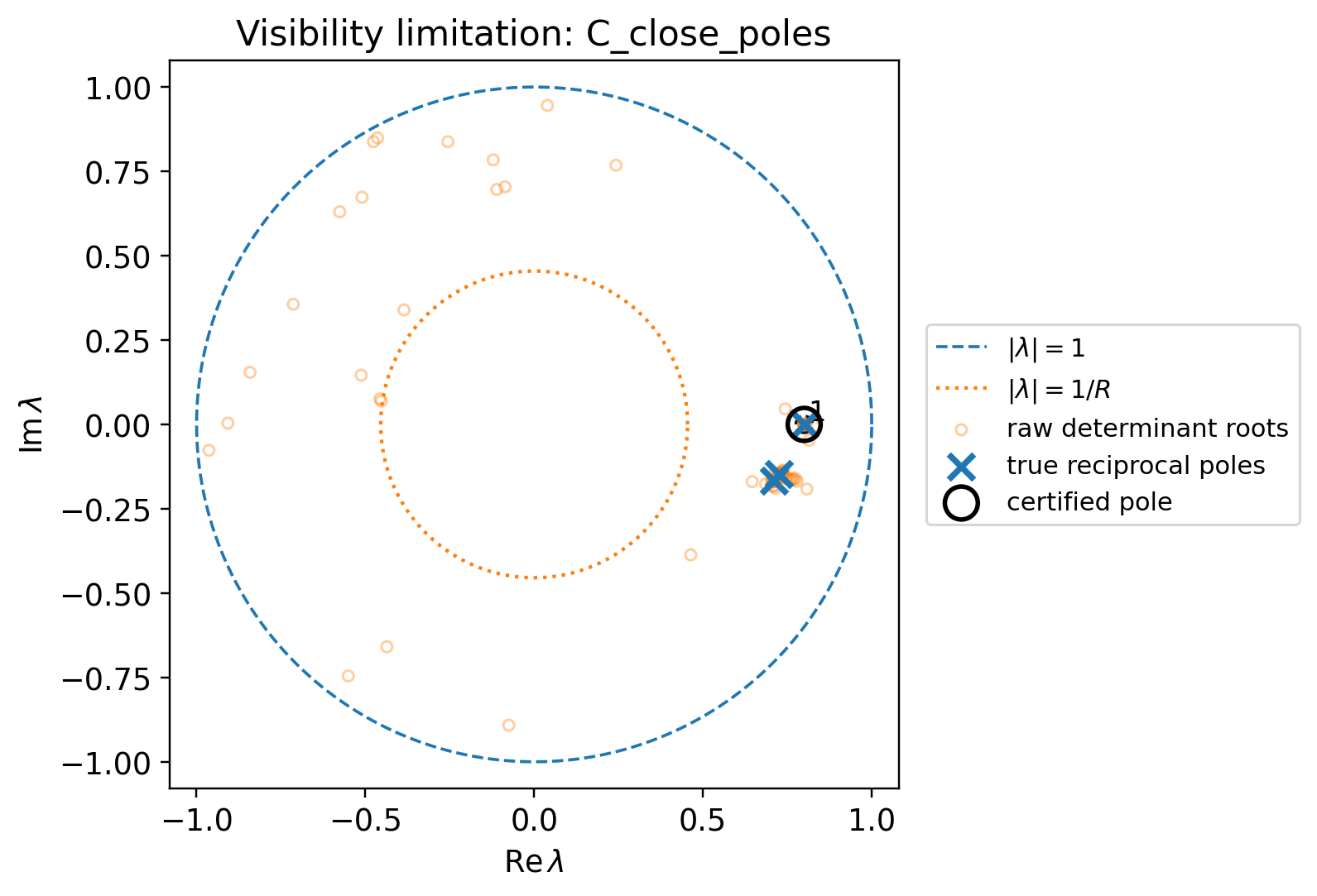}\\
{\small (c) Close poles}
\end{minipage}
\caption{Visibility limitations.  Small residues reduce persistence, poles near
the boundary of \(U_R\) have smaller admissible certifying contours, and close
poles may fail to produce separated certified components.}
\label{fig:visibility-limitations}
\end{figure}

These limitation tests support the intended interpretation of the method: the
certified set consists of poles that are visible through stable determinant
root clustering and local contour-count certification.

\section{Conclusions}
\label{sec:conclusions}

We developed a determinant-characteristic and argument-principle framework for
certifying visible exterior poles in outward meromorphic continuation from
circular boundary data.  In the pure finite-pole model, the shifted determinant
characteristic for the correct order factors exactly into a nonzero constant
times the polynomial whose zeros are the reciprocal exterior poles.  The same
node set is retained for noiseless equispaced discrete Fourier coefficients;
sampling changes only the amplitudes through an aliasing factor.  This provides
the algebraic basis for using determinant roots as candidate reciprocal poles.

For perturbed data, roots of a single empirical determinant are not reliable by
themselves.  The proposed procedure therefore follows the rule that roots
propose and contours certify.  Candidate roots are clustered in the reciprocal
search region, and local argument-principle counts are used to certify clusters
that contain one determinant zero persistently over a family of determinant
orders and shifts.  The accompanying contour moments provide local pole
estimates, while empirical margins serve as numerical safeguards against
unstable contours.

The perturbation analysis gives a sufficient Rouch\'e-type condition for stable
local zero counts.  It also clarifies the meaning of visibility: residues,
reciprocal-pole separation, distance to the boundary of the target annulus,
coefficient shifts, holomorphic background decay, and noise level all enter the
balance between determinant perturbation and pure-pole contour margin.  The
numerical experiments reflect these mechanisms.  Low-complexity configurations
are certified accurately under background and noise, whereas high noise, weak
residues, boundary-near poles, and close poles reduce persistence or lead to
partial recovery.

The method should therefore be viewed as a certification procedure for visible
poles rather than as an unconditional all-pole recovery algorithm.  Future work
will focus on sharper quantitative visibility conditions, adaptive contour and
shift selection, and systematic comparison with Hankel pencil, Pad\'e, and AAA
rational continuation methods.

\bibliography{sn-bibliography}

\end{document}